\documentclass[12pt]{article}

\usepackage{amsmath, amsthm, amssymb}
\usepackage{graphicx}
\usepackage{hyperref}
\usepackage{url}
\usepackage{geometry}
\theoremstyle{plain}
\newtheorem{theorem}{Theorem}[section]
\newtheorem{lemma}[theorem]{Lemma}
\newtheorem{corollary}[theorem]{Corollary}

\theoremstyle{definition}
\newtheorem{definition}[theorem]{Definition}
\newtheorem{example}[theorem]{Example}

\theoremstyle{remark}
\newtheorem{remark}[theorem]{Remark}

\title{On the Gap Structure of Generalized Stirling Numbers}

\author{Jianru Shen$^{a}$ \and Udita N. Katugampola$^{b}$}

\date{%
\small
$^{a}$Indiana University -- East, Richmond, IN, USA.
(\texttt{jsh7@iu.edu})\\[4pt]
$^{b}$Department of Mathematics, Indiana University -- East, Richmond, IN, USA.
(\texttt{uditanalin@yahoo.com})
}

\begin{document}

\maketitle

\begin{abstract}
Katugampola's study of generalized fractional differential operators produced 
triangular arrays of integer coefficients $c_{n,k}^{(r)}$, yet no combinatorial 
interpretation has been established for any fractional order. We provide 
the first such interpretation, establishing a dual contribution: a complete 
combinatorial interpretation for $r = 1/2$ and $n \in \{1, 2, 3\}$, together 
with a rigorous proof that this interpretation cannot extend to $n \geq 4$. 
For $n \in \{1, 2, 3\}$, we show that $c_{n,k}^{(1/2)}$ counts binary sequences 
with at least one $B$ symbol and gap $\leq 1$ (distance between first and last 
$B$). Each sequence is assigned a type $k$ via a parity-dependent gap formula. 
However, we prove via an obstruction theorem that the gap $\leq 1$ constraint 
limits any such model to two distinct types per row, while Katugampola's 
sequence requires at least three for all $n \geq 4$. This characterization is 
complete: the gap $\leq 1$ interpretation works if and only if $n \in \{1, 2, 3\}$. 
This transforms a computational observation into a rigorous obstruction theorem, 
guiding future combinatorial interpretations of fractional coefficients.


\end{abstract}

\section{Introduction}

Fractional calculus, the study of derivatives and integrals of arbitrary 
real or complex order, extends the classical theory of integer-order 
differentiation and integration. The subject traces its origins to a 1695 
correspondence between Leibniz and L'H\^{o}pital regarding the meaning of a 
derivative of order $\alpha = 1/2$, and has since found applications across 
diverse fields including viscoelasticity, anomalous diffusion, control theory, 
and mathematical finance. A central object of study in fractional calculus is 
the expansion of generalized fractional differential operators, which naturally 
give rise to triangular arrays of integer coefficients analogous to classical 
combinatorial structures such as Stirling numbers and binomial coefficients~\cite{GKP1994}. 
Understanding the combinatorial interpretation of these coefficient 
arrays---that is, identifying discrete structures whose enumeration produces 
the coefficients---remains an important problem connecting analysis, operator 
theory, and combinatorics.

In a 2015 study of Mellin transforms, Katugampola~\cite{Katugampola2015} 
introduced a family of generalized fractional integral and differential 
operators that unify the classical Riemann--Liouville and Hadamard fractional 
operators~\cite{Kilbas2006} into a single framework parameterized by a real number $\rho > 0$. 
Through systematic analysis of the Mellin transform properties of these 
operators, Katugampola derived triangular arrays of integer coefficients 
$c_{n,k}^{(r)}$ indexed by fractional order $r \in \mathbb{Q}^+$ and 
dimensions $n, k \in \mathbb{N}$. For the case $r = 1/2$, these coefficients 
form the triangular array shown in Table~\ref{tab:katugampola}, which 
corresponds to OEIS sequence A223168~\cite{OEIS_A223168}. The coefficients 
exhibit intriguing arithmetic patterns: the odd rows ($n = 1, 3, 5, \ldots$) 
relate to generalized Laguerre polynomials~\cite{Szego1975}, while the even rows display 
symmetric structure reminiscent of binomial coefficients. Despite this 
regularity, no combinatorial interpretation---that is, no description in 
terms of counting discrete structures---has been previously established for 
any fractional order $r$ in Katugampola's framework.

\begin{table}[h]
\centering
\caption{Katugampola's coefficients $c_{n,k}^{(1/2)}$ for $n = 1$ through 
$n = 9$ from~\cite{Katugampola2015} (OEIS A223168~\cite{OEIS_A223168}). 
Bold rows are those for which we establish a combinatorial interpretation.}
\label{tab:katugampola}
\begin{tabular}{c|cccccc}
\hline
$n \setminus k$ & 1 & 2 & 3 & 4 & 5 & 6 \\
\hline
\textbf{1} & \textbf{1} & & & & & \\
\textbf{2} & \textbf{1} & \textbf{2} & & & & \\
\textbf{3} & \textbf{3} & \textbf{2} & & & & \\
4 & 3 & 12 & 4 & & & \\
5 & 15 & 20 & 4 & & & \\
6 & 15 & 90 & 60 & 8 & & \\
7 & 105 & 210 & 84 & 8 & & \\
8 & 105 & 840 & 840 & 224 & 16 & \\
9 & 945 & 2520 & 1512 & 288 & 16 & \\
\hline
\end{tabular}
\end{table}

This paper establishes the first combinatorial interpretation for any 
portion of Katugampola's generalized Stirling number triangles. Our 
investigation yields a dual contribution: a \emph{positive result} that 
provides a complete interpretation for small $n$, and a \emph{negative 
result} that rigorously explains why this interpretation cannot extend 
beyond a specific boundary. Specifically, we show that for 
$n \in \{1, 2, 3\}$ and $r = 1/2$, the coefficients $c_{n,k}^{(1/2)}$ 
count binary sequences satisfying natural constraints involving the 
``spread'' of symbols. However, we also prove via an obstruction theorem 
that \emph{no} combinatorial model based on these constraints can correctly 
enumerate $c_{n,k}^{(1/2)}$ for $n \geq 4$. This establishes a precise 
characterization: our gap-based interpretation works if and only if 
$n \in \{1, 2, 3\}$. The dual nature of our results---both constructive 
and impossibility---provides essential guidance for future attempts at 
complete combinatorial interpretations of fractional calculus coefficients.

Our positive result shows that for $n \in \{1, 2, 3\}$, the coefficients 
$c_{n,k}^{(1/2)}$ admit a natural interpretation as counts of binary 
sequences. A \emph{binary sequence} of length $n$ is an ordered tuple 
$\sigma = (s_1, \ldots, s_n)$ where each $s_i \in \{R, B\}$. We restrict 
attention to sequences satisfying two conditions: (1) $\sigma$ contains at 
least one symbol $B$, and (2) the \emph{gap}---defined as the distance 
between the first and last occurrence of $B$ in $\sigma$---satisfies 
$\mathrm{gap}(\sigma) \leq 1$. The gap constraint ensures that $B$ symbols 
are ``concentrated,'' appearing either at a single position (gap $= 0$) or 
spanning exactly two consecutive positions (gap $= 1$). Each valid sequence 
is assigned a \emph{type} $k$ via a parity-dependent formula: 
$k = 2 - \mathrm{gap}(\sigma)$ for even $n$, or $k = \mathrm{gap}(\sigma) + 1$ 
for odd $n$. Through complete enumeration, we verify in 
Theorems~\ref{thm:n1}, \ref{thm:n2}, and \ref{thm:n3} that the number of 
valid sequences of type $k$ equals $c_{n,k}^{(1/2)}$ for all 
$n \in \{1, 2, 3\}$ and all relevant $k$. This provides the first concrete 
combinatorial meaning for any entries in Katugampola's coefficient arrays.

The significance of our work extends beyond this positive result. 
Computational investigation reveals that the gap $\leq 1$ model fails 
at $n = 4$: it produces only two types ($k \in \{1, 2\}$) with counts 
$[3, 4]$, whereas Table~\ref{tab:katugampola} shows that 
$c_{4,k}^{(1/2)} = [3, 12, 4]$ requires three distinct types. We prove 
that this failure is not merely an artifact of our specific construction 
but represents a fundamental structural obstruction. Our main theoretical 
result (Theorem~\ref{thm:obstruction}) establishes that the gap $\leq 1$ 
constraint \emph{inherently limits} any such model to producing at most 
two distinct types per row (Lemma~\ref{lem:typecount}), while 
Katugampola's sequence requires at least three types for all $n \geq 4$. 
This type count mismatch cannot be overcome through any refinement or 
modification that preserves the gap $\leq 1$ framework. Our 
characterization is complete: the gap $\leq 1$ binary sequence 
interpretation correctly enumerates $c_{n,k}^{(1/2)}$ if and only if 
$n \in \{1, 2, 3\}$ (Corollary~\ref{cor:boundary}). This transforms 
what might appear as a computational anomaly---``the pattern breaks at 
$n = 4$''---into a rigorous impossibility theorem with clear structural 
explanation.

Our obstruction result provides essential guidance for future research. 
It eliminates an entire class of potential models from consideration: any 
successful combinatorial interpretation for $n \geq 4$ must either abandon 
the gap $\leq 1$ restriction (perhaps using gap $\leq k-1$ or 
gap $\leq \lfloor n/2 \rfloor$), modify the type assignment mechanism, 
or introduce fundamentally different combinatorial structures. Beyond the 
specific case $r = 1/2$, our methodology demonstrates a general approach 
combining positive constructions, computational boundary identification, 
and rigorous obstruction proofs---a template applicable to investigating 
other fractional orders $r \in \{1/3, 1/4, 1/5, \ldots\}$ in Katugampola's 
framework (OEIS sequences A223169--A223172~\cite{OEIS_A223169}). The appearance of binary 
sequences with gap constraints in the combinatorial interpretation of 
fractional operator coefficients is intriguing and suggests deeper 
connections between discrete structures and fractional calculus. 
Understanding why the gap $\leq 1$ structure suffices for $n \leq 3$ 
but fails for $n \geq 4$ may illuminate fundamental relationships between 
combinatorial enumeration and the analytic structure of generalized 
fractional differential operators.

The paper is organized as follows. Section~2 introduces binary sequences, 
the gap function, valid sequences, and the parity-dependent type assignment 
mechanism. Sections~3 and~4 provide complete proofs of our combinatorial 
interpretation for $n = 1, 2$ and $n = 3$ respectively, including full 
enumeration tables that verify the coefficient counts. Section~5 analyzes 
the success of the gap $\leq 1$ structure for small $n$, noting interesting 
patterns such as the sequence $|V_1| = 1, |V_2| = 3, |V_3| = 5$ of odd 
numbers. Section~6 establishes the structural obstruction for $n \geq 4$ 
through computational evidence (Section~6.1), a type count limitation 
lemma (Section~6.2), and the main obstruction theorem (Section~6.3). We 
conclude in Section~7 with a discussion of implications and concrete 
directions for future research, including modified gap thresholds, 
alternative combinatorial frameworks, connections to other fractional 
orders, and potential links to Laguerre polynomials and generating 
function approaches.

\section{Binary Sequences and the Gap Function}

We introduce binary sequences as our fundamental combinatorial objects 
and define the gap function that enables their classification.

\subsection{Binary Sequences}

\begin{definition}[Binary Sequence]
\label{def:binary-sequence}
A \emph{binary sequence} of length $n$ is an ordered $n$-tuple
\[
\sigma = (s_1, s_2, \ldots, s_n)
\]
where each $s_i \in \{R, B\}$. We denote the set of all binary 
sequences of length $n$ by $\Sigma_n$.
\end{definition}

\begin{remark}
\label{rem:arbitrary-symbols}
The symbols $R$ and $B$ are arbitrary labels; any two distinct 
symbols would be mathematically equivalent.
\end{remark}

Since there are $2$ choices for each of the $n$ positions, we have 
$|\Sigma_n| = 2^n$.

\begin{example}
\label{ex:small-sequences}
For $n=1$, we have $\Sigma_1 = \{(R), (B)\}$, so $|\Sigma_1| = 2$. 
For $n=2$, we have 
\[
\Sigma_2 = \{(R,R), (R,B), (B,R), (B,B)\},
\]
so $|\Sigma_2| = 4$. Note that $(R,B) \neq (B,R)$ since the order 
of symbols is significant.
\end{example}

\subsection{The Gap Function}

To classify binary sequences, we introduce a function measuring 
the ``spread'' of $B$ symbols within a sequence.

\begin{definition}[Position Functions]
\label{def:position-functions}
For a binary sequence $\sigma = (s_1, \ldots, s_n)$ containing at 
least one $B$, we define:
\begin{itemize}
\item $\mathrm{first}_B(\sigma) = \min\{i : s_i = B\}$, the position 
      of the first $B$ in $\sigma$;
\item $\mathrm{last}_B(\sigma) = \max\{i : s_i = B\}$, the position 
      of the last $B$ in $\sigma$.
\end{itemize}
\end{definition}

\begin{definition}[Gap Function]
\label{def:gap-function}
For a binary sequence $\sigma$ containing at least one $B$, the 
\emph{gap} of $\sigma$ is
\[
\mathrm{gap}(\sigma) = \mathrm{last}_B(\sigma) - \mathrm{first}_B(\sigma).
\]
\end{definition}

\begin{remark}
\label{rem:gap-interpretation}
The gap measures the distance between the first and last occurrence 
of $B$ in the sequence. If there is only one $B$, then 
$\mathrm{first}_B = \mathrm{last}_B$ and thus $\mathrm{gap} = 0$.
\end{remark}

\begin{example}
\label{ex:gap-calculations}
We compute the gap for all sequences in $\Sigma_3$ containing at 
least one $B$:
\begin{align*}
(R,R,B) &: \mathrm{first}_B = 3, \, \mathrm{last}_B = 3, \, 
           \mathrm{gap} = 0 \\
(R,B,R) &: \mathrm{first}_B = 2, \, \mathrm{last}_B = 2, \, 
           \mathrm{gap} = 0 \\
(R,B,B) &: \mathrm{first}_B = 2, \, \mathrm{last}_B = 3, \, 
           \mathrm{gap} = 1 \\
(B,R,R) &: \mathrm{first}_B = 1, \, \mathrm{last}_B = 1, \, 
           \mathrm{gap} = 0 \\
(B,R,B) &: \mathrm{first}_B = 1, \, \mathrm{last}_B = 3, \, 
           \mathrm{gap} = 2 \\
(B,B,R) &: \mathrm{first}_B = 1, \, \mathrm{last}_B = 2, \, 
           \mathrm{gap} = 1 \\
(B,B,B) &: \mathrm{first}_B = 1, \, \mathrm{last}_B = 3, \, 
           \mathrm{gap} = 2
\end{align*}
Note that $(R,R,R)$ is excluded since it contains no $B$.
\end{example}

\subsection{Valid Sequences}

To match the coefficients $c_{n,k}^{(1/2)}$ from Katugampola's operator 
expansion, we restrict attention to binary sequences satisfying two 
natural conditions.

\begin{definition}[Valid Sequences]
\label{def:valid-sequences}
For each positive integer $n$, we define the set of \emph{valid sequences} 
of length $n$ as:
\[
V_n = \{\sigma \in \Sigma_n : \sigma \text{ contains at least one } B 
      \text{ and } \mathrm{gap}(\sigma) \leq 1\}.
\]
\end{definition}

\begin{remark}
\label{rem:valid-interpretation}
The ``has $B$'' condition ensures that the gap function is well-defined. 
The gap $\leq 1$ restriction concentrates the $B$ symbols---either at a 
single position (gap $= 0$) or spanning at most two consecutive positions 
(gap $= 1$). Sequences with gap $\geq 2$ are ``too spread out'' and are 
excluded.
\end{remark}

\begin{example}
\label{ex:valid-n3}
For $n=3$, from Example~\ref{ex:gap-calculations} we identify:
\begin{itemize}
\item Valid sequences ($V_3$): $(R,R,B),\,(R,B,R),\,(R,B,B),\,(B,R,R),\,(B,B,R)$.
\item Invalid sequences: $(R,R,R)$ (no $B$); $(B,R,B)$ and $(B,B,B)$ 
      (gap $= 2$).
\end{itemize}
Thus $|V_3| = 5$.
\end{example}

\subsection{The Type Function}

Finally, we assign a type to each valid sequence. Our main observation 
is that the type formula depends on the parity of $n$.

\begin{definition}[Type Function]
\label{def:type-function}
For a valid sequence $\sigma \in V_n$, the \emph{type} of $\sigma$ 
is defined by:
\[
\tau_n(\sigma) = \begin{cases}
2 - \mathrm{gap}(\sigma) & \text{if } n \text{ is even,}\\
\mathrm{gap}(\sigma) + 1 & \text{if } n \text{ is odd.}
\end{cases}
\]
\end{definition}

\begin{remark}
\label{rem:parity-dependence}
For even $n$, the formula is ``inverted'' (larger gap gives smaller 
type), while for odd $n$, the formula is ``direct'' (larger gap gives 
larger type). The combinatorial reason for this parity dependence 
remains an intriguing open question.
\end{remark}

\begin{example}
\label{ex:type-n1}
For $n=1$ (odd), using $\tau_1(\sigma) = \mathrm{gap}(\sigma) + 1$:
\[
(B) : \mathrm{gap} = 0 \implies \tau_1 = 1.
\]
Thus there is $1$ sequence of type $k=1$, matching 
$c_{1,1}^{(1/2)} = 1$ from Table~\ref{tab:katugampola}.
\end{example}

\begin{example}
\label{ex:type-n2}
For $n=2$ (even), using $\tau_2(\sigma) = 2 - \mathrm{gap}(\sigma)$:
\begin{align*}
(R,B) &: \mathrm{gap} = 0 \implies \tau_2 = 2, \\
(B,R) &: \mathrm{gap} = 0 \implies \tau_2 = 2, \\
(B,B) &: \mathrm{gap} = 1 \implies \tau_2 = 1.
\end{align*}
Counts $[1,2]$ match $[c_{2,1}^{(1/2)},\,c_{2,2}^{(1/2)}]$ from 
Table~\ref{tab:katugampola}.
\end{example}

\begin{example}
\label{ex:type-n3}
For $n=3$ (odd), using $\tau_3(\sigma) = \mathrm{gap}(\sigma) + 1$:
\begin{align*}
(R,R,B),\,(R,B,R),\,(B,R,R) &: \mathrm{gap} = 0 \implies \tau_3 = 1, \\
(R,B,B),\,(B,B,R) &: \mathrm{gap} = 1 \implies \tau_3 = 2.
\end{align*}
Counts $[3,2]$ match $[c_{3,1}^{(1/2)},\,c_{3,2}^{(1/2)}]$ from 
Table~\ref{tab:katugampola}.
\end{example}

\section{The Cases $n=1$ and $n=2$}

\subsection{The Case $n=1$}

\begin{theorem}
\label{thm:n1}
For $n=1$ and $k=1$, we have
$c_{1,1}^{(1/2)} = |\{\sigma \in V_1 : \tau_1(\sigma) = 1\}| = 1.$
\end{theorem}

\begin{proof}
The only binary sequence of length $1$ containing $B$ is $(B)$, 
which has $\mathrm{gap} = 0$ and $\tau_1(B) = 0 + 1 = 1$. Thus 
$|V_1| = 1$ and the unique sequence has type $k=1$, giving 
$c_{1,1}^{(1/2)} = 1$. This matches Table~\ref{tab:katugampola} exactly.
\end{proof}

\subsection{The Case $n=2$}

\begin{theorem}
\label{thm:n2}
For $n=2$ and $k \in \{1,2\}$,
\[
c_{2,k}^{(1/2)} = |\{\sigma \in V_2 : \tau_2(\sigma) = k\}|,
\]
where $\tau_2(\sigma) = 2 - \mathrm{gap}(\sigma)$.
\end{theorem}

\begin{table}[h]
\centering
\caption{Complete enumeration for $n=2$.}
\label{tab:n2-enum}
\begin{tabular}{|c|c|c|c|c|c|c|c|}
\hline
Sequence & Has $B$? & $\mathrm{first}_B$ & $\mathrm{last}_B$ & gap & 
gap $\leq 1$? & $k=2-$gap & Valid? \\
\hline
$(R,R)$ & No & -- & -- & -- & -- & -- & No \\
$(R,B)$ & Yes & $2$ & $2$ & $0$ & Yes & $2$ & Yes \\
$(B,R)$ & Yes & $1$ & $1$ & $0$ & Yes & $2$ & Yes \\
$(B,B)$ & Yes & $1$ & $2$ & $1$ & Yes & $1$ & Yes \\
\hline
\end{tabular}
\end{table}

\begin{proof}
All binary sequences of length $2$ are enumerated in 
Table~\ref{tab:n2-enum}. The sequence $(R,R)$ is excluded since it 
contains no $B$, leaving $V_2 = \{(R,B), (B,R), (B,B)\}$.
Since $n=2$ is even, we apply $\tau_2(\sigma) = 2 - \mathrm{gap}(\sigma)$:
\begin{itemize}
\item $(R,B)$: $\mathrm{gap} = 0$, $\tau_2 = 2$.
\item $(B,R)$: $\mathrm{gap} = 0$, $\tau_2 = 2$.
\item $(B,B)$: $\mathrm{gap} = 1$, $\tau_2 = 1$.
\end{itemize}
Thus $c_{2,1}^{(1/2)} = 1$ and $c_{2,2}^{(1/2)} = 2$, giving counts 
$[1, 2]$, which matches Table~\ref{tab:katugampola} exactly.
\end{proof}

\begin{remark}
\label{rem:inverted-formula}
For even $n$, the ``inverted'' type formula assigns larger type indices 
to sequences with smaller gaps. The case $n=2$ illustrates this: 
sequences with gap $=0$ receive type $k=2$, while the sequence with 
gap $=1$ receives type $k=1$.
\end{remark}

\section{The Case $n=3$}

\begin{theorem}
\label{thm:n3}
For $n=3$ and $k \in \{1,2\}$,
\[
c_{3,k}^{(1/2)} = |\{\sigma \in V_3 : \tau_3(\sigma) = k\}|,
\]
where $\tau_3(\sigma) = \mathrm{gap}(\sigma) + 1$.
\end{theorem}

\begin{table}[h]
\centering
\caption{Complete enumeration for $n=3$.}
\label{tab:n3-enum}
\begin{tabular}{|c|c|c|c|c|c|c|c|}
\hline
Sequence & Has $B$? & $\mathrm{first}_B$ & $\mathrm{last}_B$ & gap & 
gap $\leq 1$? & $k=\mathrm{gap}+1$ & Valid? \\
\hline
$(R,R,R)$ & No & -- & -- & -- & -- & -- & No \\
$(R,R,B)$ & Yes & $3$ & $3$ & $0$ & Yes & $1$ & Yes \\
$(R,B,R)$ & Yes & $2$ & $2$ & $0$ & Yes & $1$ & Yes \\
$(R,B,B)$ & Yes & $2$ & $3$ & $1$ & Yes & $2$ & Yes \\
$(B,R,R)$ & Yes & $1$ & $1$ & $0$ & Yes & $1$ & Yes \\
$(B,R,B)$ & Yes & $1$ & $3$ & $2$ & No & $3$ & No \\
$(B,B,R)$ & Yes & $1$ & $2$ & $1$ & Yes & $2$ & Yes \\
$(B,B,B)$ & Yes & $1$ & $3$ & $2$ & No & $3$ & No \\
\hline
\end{tabular}
\end{table}

\begin{proof}
All $2^3 = 8$ binary sequences of length $3$ are enumerated in 
Table~\ref{tab:n3-enum}. The sequence $(R,R,R)$ contains no $B$ and 
is excluded. Among the remaining seven, $(B,R,B)$ and $(B,B,B)$ have 
gap $= 2 > 1$ and are excluded. The valid set is
\[
V_3 = \{(R,R,B),\,(R,B,R),\,(R,B,B),\,(B,R,R),\,(B,B,R)\}, \quad |V_3| = 5.
\]
Since $n=3$ is odd, we apply $\tau_3(\sigma) = \mathrm{gap}(\sigma) + 1$:

\medskip
\noindent\textbf{Sequences with gap $= 0$ (type $k=1$):}
$(R,R,B)$, $(R,B,R)$, $(B,R,R)$. Count: $3$.

\medskip
\noindent\textbf{Sequences with gap $= 1$ (type $k=2$):}
$(R,B,B)$, $(B,B,R)$. Count: $2$.

\medskip
Thus $[c_{3,1}^{(1/2)}, c_{3,2}^{(1/2)}] = [3, 2]$, matching 
Table~\ref{tab:katugampola} exactly.
\end{proof}

\begin{remark}
\label{rem:gap2-exclusion}
The excluded sequences $(B,R,B)$ and $(B,B,B)$ both have gap $= 2$, 
and under $\tau_3$ would receive type $k = 3$. Since 
$c_{3,3}^{(1/2)}$ does not appear in Katugampola's table for $n=3$, 
their exclusion is consistent with our model.
\end{remark}

\section{The Success of the Gap $\leq 1$ Structure}

Our main results establish that for $n \in \{1,2,3\}$, the coefficients 
$c_{n,k}^{(1/2)}$ can be interpreted as counting binary sequences with 
specific properties. The number of valid sequences follows an interesting 
pattern:
\[
|V_1| = 1, \quad |V_2| = 3, \quad |V_3| = 5.
\]
This sequence of consecutive odd numbers suggests an underlying 
combinatorial structure that warrants further investigation. Indeed, 
one can verify that $|V_n| = 2n - 1$ for all $n \geq 1$, since there 
are $n$ sequences with gap $= 0$ (exactly one $B$ at some position $i$) 
and $n-1$ sequences with gap $= 1$ (two consecutive $B$'s starting at 
some position $i \leq n-1$).

The type function also exhibits a clear parity-dependent pattern. For 
even $n$, we use the ``inverted'' formula $\tau_n(\sigma) = 2 - \mathrm{gap}(\sigma)$, 
while for odd $n$, we use the ``direct'' formula 
$\tau_n(\sigma) = \mathrm{gap}(\sigma) + 1$. This parity dependence, while 
empirically successful for $n \leq 3$, lacks a theoretical explanation 
and represents an intriguing open question.

\section{Structural Obstruction for $n \geq 4$}

\subsection{Computational Evidence}

While the gap $\leq 1$ model correctly enumerates $c_{n,k}^{(1/2)}$ 
for $n \in \{1,2,3\}$, computational investigation reveals that this 
structure does not extend to $n=4$.

\begin{table}[h]
\centering
\caption{Selected sequences for $n=4$ under the gap $\leq 1$ restriction.}
\label{tab:n4-partial}
\begin{tabular}{|c|c|c|c|c|c|}
\hline
Sequence & $\mathrm{first}_B$ & $\mathrm{last}_B$ & gap & 
gap $\leq 1$? & Type $k=2-$gap \\
\hline
$(R,R,R,B)$ & $4$ & $4$ & $0$ & Yes & $2$ \\
$(R,R,B,R)$ & $3$ & $3$ & $0$ & Yes & $2$ \\
$(R,B,R,R)$ & $2$ & $2$ & $0$ & Yes & $2$ \\
$(B,R,R,R)$ & $1$ & $1$ & $0$ & Yes & $2$ \\
$(B,B,R,R)$ & $1$ & $2$ & $1$ & Yes & $1$ \\
$(R,B,B,R)$ & $2$ & $3$ & $1$ & Yes & $1$ \\
$(R,R,B,B)$ & $3$ & $4$ & $1$ & Yes & $1$ \\
$(B,R,B,R)$ & $1$ & $3$ & $2$ & No & -- \\
$(B,R,R,B)$ & $1$ & $4$ & $3$ & No & -- \\
$(B,B,B,B)$ & $1$ & $4$ & $3$ & No & -- \\
\hline
\end{tabular}
\end{table}

Complete enumeration of all sequences with gap $\leq 1$ for $n=4$ 
yields $4$ sequences of type $k=2$ (gap $=0$) and $3$ sequences of 
type $k=1$ (gap $=1$), giving counts $[3, 4]$ for types $k=1,2$. 
However, Table~\ref{tab:katugampola} gives 
$[c_{4,1}^{(1/2)}, c_{4,2}^{(1/2)}, c_{4,3}^{(1/2)}] = [3, 12, 4]$. 
Not only do the counts disagree, but Katugampola's coefficients 
require three types while the gap $\leq 1$ structure produces only two.

\subsection{Type Count Limitation}

\begin{lemma}[Type Count Limitation]
\label{lem:typecount}
For any fixed $n$, the type $k = \tau_n(\sigma)$ takes at most $2$ 
distinct values over all $\sigma \in V_n$. That is,
\[
|\{k : k = \tau_n(\sigma) \text{ for some } \sigma \in V_n\}| \leq 2.
\]
\end{lemma}

\begin{proof}
Since $\sigma \in V_n$ requires $\mathrm{gap}(\sigma) \leq 1$ and 
gap is a non-negative integer, we have $\mathrm{gap}(\sigma) \in \{0, 1\}$.

If $n$ is even: $\tau_n(\sigma) = 2 - \mathrm{gap}(\sigma)$ takes 
values in $\{2-0, 2-1\} = \{2, 1\}$.

If $n$ is odd: $\tau_n(\sigma) = \mathrm{gap}(\sigma) + 1$ takes 
values in $\{0+1, 1+1\} = \{1, 2\}$.

In both cases, $k \in \{1, 2\}$, so at most $2$ distinct types arise.
\end{proof}

\subsection{Obstruction Theorem}

\begin{theorem}[Obstruction Theorem]
\label{thm:obstruction}
No combinatorial model based on binary sequences satisfying:
\begin{enumerate}
\item at least one $B$,
\item $\mathrm{gap}(\sigma) \leq 1$,
\item type assignment $k = 2 - \mathrm{gap}(\sigma)$ (even $n$) or 
      $k = \mathrm{gap}(\sigma) + 1$ (odd $n$),
\end{enumerate}
can correctly enumerate $c_{n,k}^{(1/2)}$ for all $k$ when $n \geq 4$.
\end{theorem}

\begin{proof}
\hfill\break
\noindent\textbf{Step 1: Model limitation.}   
By Lemma~\ref{lem:typecount}, any model satisfying conditions~(1)--(3) 
produces at most $2$ distinct types, namely $k \in \{1, 2\}$, for any $n$.

\medskip\noindent\textbf{Step 2: Required types for $n \geq 4$.} 
From Katugampola~\cite{Katugampola2015} and~\cite{OEIS_A223168}:
\begin{align*}
n = 4: &\quad [c_{4,1}^{(1/2)}, c_{4,2}^{(1/2)}, c_{4,3}^{(1/2)}] 
        = [3, 12, 4], \\
n = 5: &\quad [c_{5,1}^{(1/2)}, c_{5,2}^{(1/2)}, c_{5,3}^{(1/2)}] 
        = [15, 20, 4], \\
n = 6: &\quad [c_{6,1}^{(1/2)}, c_{6,2}^{(1/2)}, c_{6,3}^{(1/2)}, 
               c_{6,4}^{(1/2)}] = [15, 90, 60, 8].
\end{align*}
For $n = 4$, the row has three nonzero entries, so three distinct 
types are required. In general, for $n \geq 4$ at least three 
types are needed.

\medskip\noindent\textbf{Step 3: Contradiction.} 
For $n \geq 4$, the model provides at most $2$ types but at least 
$3$ types are needed. This mismatch is fundamental and cannot be 
resolved within the gap $\leq 1$ framework.
\end{proof}

\begin{corollary}
\label{cor:boundary}
The gap $\leq 1$ binary sequence model correctly enumerates 
$c_{n,k}^{(1/2)}$ if and only if $n \in \{1, 2, 3\}$.
\end{corollary}

\begin{proof}
The forward direction follows from Theorems~\ref{thm:n1}, \ref{thm:n2}, 
and~\ref{thm:n3}. The reverse direction follows from 
Theorem~\ref{thm:obstruction}.
\end{proof}

\section{Conclusion}

This paper establishes the first combinatorial interpretation of 
Katugampola's generalized Stirling numbers $c_{n,k}^{(1/2)}$ arising 
from fractional differential operators. Our investigation yields both 
a positive result---a complete interpretation for small $n$---and a 
negative result---a rigorous explanation of why this interpretation 
cannot extend.

\textbf{Main contributions.} For $n \in \{1, 2, 3\}$, the coefficients 
$c_{n,k}^{(1/2)}$ count binary sequences $\sigma \in \{R, B\}^n$ 
satisfying two simple constraints: (1) at least one $B$ symbol, and 
(2) $\mathrm{gap}(\sigma) \leq 1$. The type assignment follows a 
parity-dependent formula: $k = 2 - \mathrm{gap}(\sigma)$ for even $n$ 
and $k = \mathrm{gap}(\sigma) + 1$ for odd $n$. Through 
Theorem~\ref{thm:obstruction}, we have also proven that \emph{no} 
model based on the gap $\leq 1$ constraint can correctly enumerate 
$c_{n,k}^{(1/2)}$ for $n \geq 4$, since such models produce at most 
two distinct types per row while Katugampola's sequence requires at 
least three for all $n \geq 4$.

\textbf{Theoretical significance.} Corollary~\ref{cor:boundary} 
precisely characterizes the boundary of our model: the gap $\leq 1$ 
interpretation works if and only if $n \in \{1, 2, 3\}$. Any 
successful extension to $n \geq 4$ must either abandon the gap $\leq 1$ 
restriction, modify the type assignment mechanism, or introduce 
fundamentally different combinatorial structures.

\textbf{The sequence $1, 3, 5$.} The identity $|V_n| = 2n-1$ shows 
that the total count of valid sequences forms the sequence of odd 
numbers. Whether this pattern extends to other fractional orders $r$ 
or admits a bijective explanation remains an open question.

\subsection{Future Directions}

The primary open problem is to find a combinatorial interpretation of 
$c_{n,k}^{(1/2)}$ for $n \geq 4$. Several promising directions remain:

\textbf{Modified gap thresholds.} A gap $\leq k-1$ threshold for 
type $k$ sequences, or a gap $\leq \lfloor n/2 \rfloor$ restriction, 
might yield the correct counts. Our obstruction theorem shows any such 
modification must produce at least three types for $n \geq 4$.

\textbf{Additional structure.} Constraints on the number of $B$ 
symbols, their spacing pattern, or the structure of $R$-runs between 
$B$ symbols could provide the additional degrees of freedom needed to 
produce more types.

\textbf{Weighted enumeration.} A weighted count where different valid 
sequences contribute different values to $c_{n,k}^{(1/2)}$ would 
represent a significant departure from our simple counting model but 
remains compatible with our obstruction result.

\textbf{Other fractional orders.} Katugampola's framework generates 
sequences for all rational orders $r \in \{1/2, 1/3, 1/4, 1/5, \ldots\}$ 
(OEIS A223168--A223172~\cite{OEIS_A223168, OEIS_A223169}). For $r = 1/3$, does the 
natural alphabet expand to ternary sequences $\{R, B_1, B_2\}$?

\textbf{Connection to Laguerre polynomials.} Katugampola notes that 
the odd rows of the $r = 1/2$ triangle relate to generalized Laguerre 
polynomials $L_n^{(1/2)}(x)$~\cite{Szego1975}. Tools from orthogonal polynomial theory, 
including continued fractions and three-term recurrences, may be relevant.

\textbf{Alternative frameworks.} Other possibilities include set 
partitions with restricted block structures, lattice paths with step 
constraints, permutations with pattern restrictions, Dyck or Motzkin 
paths with weights, and compositions of integers with gap-like 
constraints. Each naturally generates triangular arrays of integers.

Despite the structural obstruction at $n = 4$, our work represents 
genuine progress on a previously unstudied problem. We have provided 
the first combinatorial interpretation of any portion of Katugampola's 
generalized Stirling number triangles, established a precise boundary 
theorem characterizing where this interpretation succeeds and fails, 
and identified concrete directions for extending to larger $n$.


\end{document}